\documentclass[a4paper]{article}
\usepackage[dvipsnames]{xcolor}
\usepackage{amsmath}
\usepackage{amssymb}
\usepackage{amsthm}
\usepackage{mathrsfs}
\usepackage{enumerate}
\usepackage{color}
\usepackage{geometry}
\usepackage{bbm}
\usepackage{slashed}
\usepackage{hyperref} 
\usepackage{enumerate,cancel,amssymb,centernot}
 \usepackage{latexsym}
 \usepackage{verbatim}
\usepackage{graphicx}
\usepackage{tikz}
\RequirePackage[numbers]{natbib}
\DeclareSymbolFont{sfoperators}{OT1}{ptm}{m}{n}
\DeclareSymbolFontAlphabet{\mathsf}{sfoperators}

\usepackage{float} 
\usepackage{amsthm}
\usepackage{graphicx} 
\usepackage{mathtools, stmaryrd}
\usepackage{ulem}
\usepackage[english]{babel}
\usepackage{todonotes}
\usepackage[T1]{fontenc}
\usepackage{enumitem}

\numberwithin{equation}{section}

\newcommand{\keywords}[1]{\par\noindent\textbf{Keywords:} #1}
\newcommand{\msc}[1]{\par\noindent\textbf{MSC Classification:} #1}

\newtheorem{thm}{Theorem}[section]

\newtheorem*{thmA*}{Theorem A}
\newtheorem*{thmB*}{Theorem B}

\newtheorem{lemma}[thm]{Lemma}

\newtheorem{proposition}[thm]{Proposition}

\def\p{{\mathbb P}}
\def\e{{\mathbb E}}

\def\r{{\mathsf R}}
\def\z{{\mathbb Z}}

\def\N{{\mathbb N}}

\newcommand{\visitoneset}[1]{\mathcal{A}_{#1}}
\newcommand{\visitonenumber}[1]{g_1(#1)}
\newcommand{\curvevisitoneset}[1]{\mathcal{D}_{#1}}
\newcommand{\curvevisitonenumber}[1]{d({#1})}
\makeatletter
\def\operator@font{\mathgroup\symsfoperators}
\makeatother

\title{The Exact Limsup Constant for Once-Visited Sites of One-Dimensional Simple Random Walk}
\author{Chenxu Feng  \\ Peking University \and Chenxu Hao\thanks{Corresponding author. E-mail: \texttt{haochenxu@csu.edu.cn}}  \\ Central South University}

\begin{document}

\maketitle

\begin{abstract}
For a one-dimensional simple random walk, let $g_1(n)$ denote the number of sites visited exactly once at time $n$. Major (1988) proved that
\begin{equation*}
\limsup_{n\to\infty}\frac{g_1(n)}{\log^2 n}=C\qquad a.s.
\end{equation*}
where $C$ is a positive and finite constant. While this result settled the question of existence, the exact value of $C$ remained unknown.

In this paper, we determine that $C=1/16$. The main novelty of our work lies in introducing a self-boosting iterative framework for analysis.
\end{abstract}

\keywords{Simple random walk, Local time, Rarely visited site.}

\msc{60F15, 60J55}

\section{Introduction}
Let $(S_n)_{n\geq 0}$ be a discrete-time simple random walk (SRW) on $\mathbb{Z}$  starting at $S_0 = 0$.  
For any $n\in\N$ and $s\in\z$, write 
\begin{equation*}
\xi(s,n)=\#\{0\le j\le n : S_j=s\},\qquad \mathcal{R}(n)=\{s\in\mathbb{Z}:\exists\ k\in[0,n],\ S_k=s\},
\end{equation*}
for the local time (number of visits) at site $s$ up to time $n$, and the range (set of visited sites) of the random walk by time $n$ respectively. Let $g_k(n)$ be the number of sites visited exactly $k$ times up to time $n$, that is, $$g_k(n):=\#\{s:\xi(s,n)=k\}.$$
Among these, $g_1$ is particularly interesting and  has been the focus of significant research interest. 
Intuitively, the set of once-visited sites can be linked to the ``points of increase'' of the SRW (see Dvoretzky, Erd\H{o}s and Kakutani \cite{DEK61} and Peres \cite{Pe96}).
An interesting result was established by Newman \cite{Ne84}, who proved that $\e\, g_1(n)=2$ for all $n\in\N$. Furthermore, Han \cite{CH25} proved the convergence $\mathbb{E}[g_k(n)] \to 2$ for any fixed $k \ge 2$.
\footnote{During C.H.'s visit to Sichuan University in September 2025, Yinshan Chang kindly informed C.H. that this convergence result was obtained in Yuhang Han's undergraduate thesis \cite{CH25} of which Y.C. is the supervisor.}

However, is it possible that as $n$ tends to infinity, $g_1(n)$ attains atypically large values at some $n$'s? Motivated by this question, Erd\H{o}s and R\'{e}v\'{e}sz (see \cite{Ma88} for details) posed the problem of identifying the scaling function  $\kappa(n)\to\infty$ for which $\frac{g_1(n)}{\kappa(n)}$ admits a non-degenerate limit, i.e.,
\begin{equation*}
\limsup\limits_{n\to\infty}\dfrac{g_1(n)}{\kappa(n)}=C\in(0,\infty)\quad\quad a.s.
\end{equation*}
This question was resolved by Major as follows.
\begin{thmA*}[\cite{Ma88}]\label{OVS-limsup-behavior-unknown-constant}
There exists a constant $0<C<\infty$ such that
\begin{eqnarray*}
\limsup_{n\to\infty} \dfrac{g_1(n)}{\log^2 n}=C~~~~~~a.s.
\end{eqnarray*}
\end{thmA*} 
The bounds $\frac{1}{256}\le C\le\frac{1}{4}$ were later established based on an argument communicated to Major by Cs\'{a}ki (see Remark 6 in \cite{Ma88}).
However, the precise value of the constant $C$ in the above theorem remained unknown. In this paper, we determine this constant exactly. Our main result is the following.

\begin{thm}\label{main result}
With probability $1$,
\begin{equation*}
\limsup_{n\to\infty} \dfrac{g_1(n)}{\log^2 n}=\frac{1}{16}.
\end{equation*}
\end{thm}

As observed in Remark 4 of Major \cite{Ma88}, for any fixed $k\ge 2$, $\limsup_{n\to\infty}g_k(n)$ is also of order $\log^2 n$; however, our method does not yield the exact leading-order constant.

In comparison, the corresponding quantity $g_1^{(d)}(n)$ for the $d$-dimensional simple random walk, $d \geq 2$.
According to Erd\H{o}s and Taylor \cite{ET60} and Flatto \cite{Fl76}, it is known that almost surely
\begin{equation*}
\lim_{n\to\infty} \dfrac{g_1^{(2)}(n)\cdot\log^2 n}{\pi^2n}=1\quad\quad\mbox{and}\quad\quad\lim_{n\to\infty}\frac{g_1^{(d)}(n)}{n}=\gamma_d^2,
\end{equation*}
where $d\ge3$ and $\gamma_d$ is the escape probability of a simple random walk on $\mathbb{Z}^d$. 

A related quantity, the minimal local time defined as $f(n)=\min\{\xi(s,n):s\in\mathcal{R}(n)\}$, was introduced by Erd\H{o}s and R\'{e}v\'{e}sz \cite{ER87,ER91} to study its $\limsup$ behavior. After a series of developments by T\'{o}th \cite{To96} and R\'{e}v\'{e}sz \cite{Re13}, our previous work \cite{FH25} ultimately established that
\begin{equation*}
\limsup_{n\to\infty} \frac{f(n)}{\log\log n} = \frac{1}{\log 2}~~~~~~a.s.
\end{equation*}
This result implies that 
$$\liminf\limits_{n\to\infty}g_1(n)=0~~~~~~a.s.$$ 

Turning our attention from the number of least visited to the most frequently visited sites (also called {\it favorite sites}), we find a rich body of results.
For instance, T\'{o}th \cite{To01} and Ding and Shen \cite{DS18} established that a SRW on $\mathbb{Z}$ almost surely has exactly three favorite sites infinitely often, but never more. Extensions to higher dimensions have been studied by Erd\H{o}s and R\'{e}v\'{e}sz \cite{ER91}, as well as by the second author of this paper together with Li, Okada and Zheng \cite{HLOZ24}. 
There are many results on the favorite sites of SRW. See Kesten \cite{Ke65}, Erd\H{o}s and R\'{e}v\'{e}sz \cite{ER84}, Bass and Griffin \cite{BG85}, Cs\'{a}ki and F\"{o}ldes \cite{CF86}, Cs\'{a}ki and Shi \cite{CS98}, Cs\'{a}ki, R\'{e}v\'{e}sz and Shi \cite{CRS00}, Lifshits and Shi \cite{LS04}, Bass \cite{Ba23}, Dembo, Peres, Rosen and Zeitouni \cite{DPRZ01}, Rosen \cite{Ro05} and for comprehensive surveys, we refer to Shi and T\'{o}th \cite{ST00} and Okada \cite{Oka16}.

Our overall strategy is to decompose the entire path of the random walk into alternating ``inward'' and ``outward'' excursions by stopping times (Inspired by our previous work \cite{FH25} and referring to \eqref{Stopping-times} for the precise definitions, we capture the transitions between exploration and return to a previous range). This allows us to reduce the original problem to analyzing a sequence of i.i.d.~excursions.
More precisely, we study the probability of the {\it rare event} that each excursion contains unusually large number of once-visited sites.
An intuitive observation is that the set of sites visited only once is closely related to the ``points of increase'' at each excursion.
We finalize the proof by employing a self-boosting iterative framework that bounds the maximum number of once-visited sites per excursion. The power of this method lies in its ability to handle both the lower and upper bounds of the probability of {\it rare events}, which are separately derived in Sections \ref{se:3} and \ref{se:4}.

We now briefly outline the organization of this paper. Section \ref{se:2} contains the necessary preliminaries. The proof of our main result, Theorem \ref{main result}, is presented in Section \ref{Proof of main result}. Finally, a key proposition for Theorem \ref{main result} is proved in Section \ref{tech-section}.

\section{Preliminaries}\label{se:2}

We denote by $c$ and $C$ positive and finite constants whose values are universal but may change from line to line.
If $\{a_n\}$ and $\{b_n\}$ are non-negative sequences, then we write
$a_n\lesssim b_n$ if there exists $c>0$ such that $a_n \leq c\, b_n$ for all $n$, and $a_n \asymp b_n$ if $a_n\lesssim b_n$ and $b_n\lesssim a_n$. For a sequence $\{a'_n\}$, we write $a'_n=O(b_n)$ if $|a'_n| \lesssim b_n$.
For a sequence $\{c_n\}$, we write $c_n=o(b_n)$ if $\lim\limits_{n\to\infty}\frac{b_n}{c_n}=\infty$.

We now turn to random walk. 
Recall that $(S_n)_{n\geq 0}$ is a random walk on $\mathbb{Z}$, starting at the origin and let $\p^S$ stand for the probability measure. Write $\p=\p^S$ for short.
We also define corresponding $\sigma$-algebras:
\begin{equation}\label{sigma}
\mathcal{F}_{n}^{S}:=\sigma\{S_{[0,n]}\},
\end{equation}
where $S_{[0,n]}$ denotes the sub-path $\{S_{0},S_{1},...,S_{n-1},S_{n}\}$.

Let $M_n = \max\limits_{0 \leq k \leq n} S_k$ and $N_n = \min\limits_{0 \leq k \leq n} S_k$ denote the maximum and minimum positions of the random walk up to time $n$, respectively, and define the span $M_{n}-N_{n}:=\r_{n}$.
Then by the Law of Iterated Logarithm (\cite[Section 4.4 and 5.3]{Re13}),
\begin{align}\label{LIL-range}
    \dfrac{\log\r_n}{\log n}=\frac12~~~~~~~~a.s.
\end{align}
We define two sequences of stopping times that alternate between ``inward'' and ``outward'' excursions of the random walk. Formally, let $\tau_1 = 1$ and $M_0 = N_0 = 0$. For $i \geq 1$, define
\begin{align}\label{Stopping-times}
    &\sigma_i = \inf \left\{ k > \tau_i : S_k \in [N_{\tau_i - 1}, M_{\tau_i - 1}] \right\}, \nonumber\\
    &\tau_{i+1} = \inf \left\{ k > \sigma_i : S_k \notin [N_{\sigma_i}, M_{\sigma_i}] \right\}.
\end{align}
\begin{figure}[htbp]
    \centering
    \includegraphics[scale=0.45]{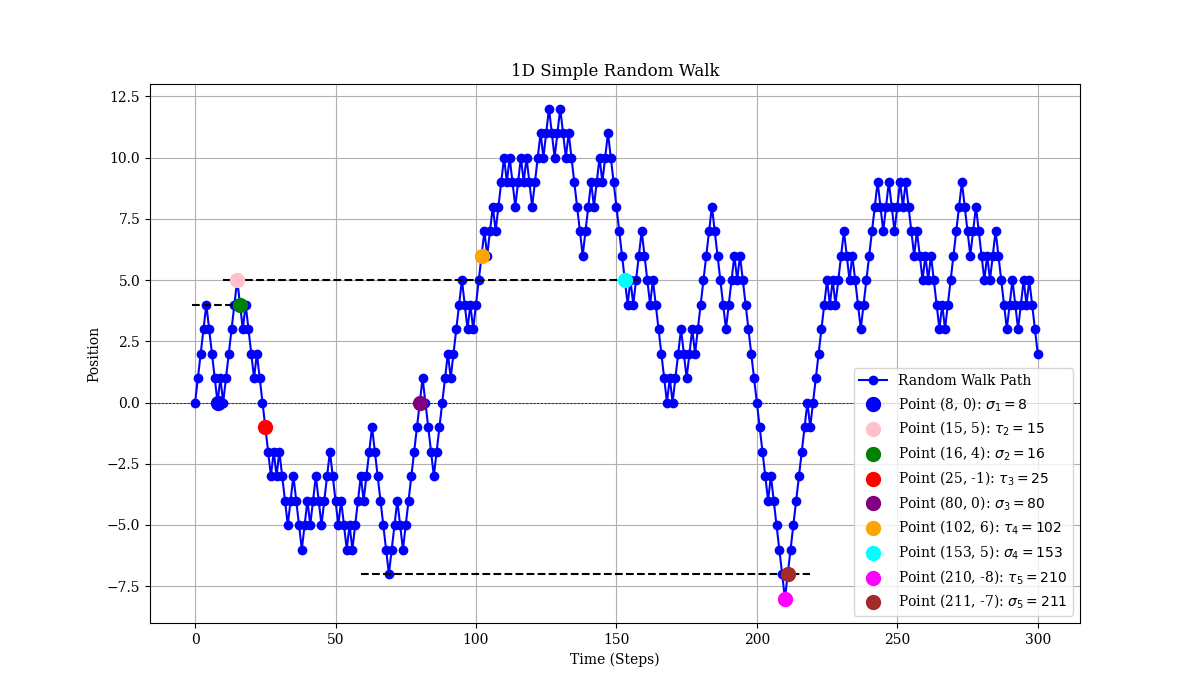} 
    \caption{Decomposition of the random walk path into inward and outward excursions using the stopping times defined in \eqref{Stopping-times}. The blue dots represent the walker's position, with selected stopping times $\sigma_i$ and $\tau_i$ marked. In this particular realization, $\tau_6 > 300$.}
    \label{fig:stopping_time} 
\end{figure}
See Figure~\ref{fig:stopping_time} for illustration. The figure shows a typical realization of the random walk path, with the stopping times $\sigma_i$ and $\tau_i$ marked. Each ``outward'' excursion between $\tau_i$ and $\sigma_i$ corresponds to the walker exploring new region, while the ``inward'' segments between $\sigma_i$ and $\tau_{i+1}$ represent returns to previously visited regions.
An important observation is that the time $\sigma_i$ can also be characterized as the first return to the pre-$\tau_i$ level:
\begin{equation*}
    \sigma_i = \inf \{ k > \tau_i : S_k = S_{\tau_i - 1} \}.
\end{equation*}

Finally, let $(T_n)$ be a SRW on $\mathbb{Z}$ starting at $T_1 = 1$ and stopped at $\sigma := \inf \{ n : T_n = 0 \}$. For each $i \geq 1$, we then define the {\it excursion process}
\begin{equation*}
T^i_n = | S_{\tau_i - 1 + n} - S_{\tau_i - 1} |, \quad \text{for } 1 \leq n \leq \sigma_i - \tau_i + 1.
\end{equation*}
By the strong Markov property applied at time $\tau_i - 1$, the process $\left( T^i_n \right)_{1 \leq n \leq \sigma_i - \tau_i + 1}$ is independent of $\mathcal{F}_{\tau_i}$ and equal in distribution to $\left( T_n \right)_{1 \leq n \leq \sigma}$. This construction allows us to analyze each outward excursion as an independent copy of a reflected random walk segment.

In the following, all asymptotic statements (e.g., involving $M, j, k, s \to \infty$) are to be understood as holding for all sufficiently large values of the indices, even when not explicitly stated.

\section{Proof of Theorem \ref{main result}}\label{Proof of main result}

We divide the proof of Theorem \ref{main result} into two lemmas and a proposition.

\begin{lemma}\label{lem1.2}
    \begin{equation*}
        \lim_{j\rightarrow \infty}\frac{\log(\sigma_j)}{\log j}=2\quad\quad\quad a.s.
    \end{equation*}
\end{lemma}
\begin{proof}
First, we connect the growth of $\sigma_j$ with the range of the random walk. Note that the increase of the range of $(S_n)$ during $[\tau_j,\sigma_j]$ is identical to the range of $(T_n)_{1\le n\le \sigma}$. Therefore, for $k\in\mathbb{Z}^+$,
\begin{equation}\label{T-distribution}
        \p\left( \r_{\sigma_{j+1}}\geq \r_{\sigma_j}+k \right)=\p\left( \max\limits_{1\leq n\leq \sigma}T_{n}\geq k\right)=\frac{1}{k}.
\end{equation}
Suppose $X_i$, $i\in \N$ are i.i.d. random variables with distribution
\begin{equation}\label{eq:logvariable}
        \p(X_i=k)=\frac{1}{k(k+1)}\text{~for~} k\in \z^+.
\end{equation}
Hence, the distribution of $\{\r_{\sigma_j}\}$ is identical to the distribution of $\left\{\sum\limits_{i=1}^{j}X_i\right\}$ with $j\in\z^+$. Furthermore, for any $C,\epsilon>0$, 
    \begin{align}
        \p\left(\sum\limits_{i=1}^{j}X_i\geq Cj^{1+\epsilon}\right)
        &\le\p\left(\mathop{\cup}_{i=1}^{j}\left\{X_i>j^{1+\epsilon}\right\}\right)+\p\left(\sum_{i=1}^{j}X_i\cdot1_{\{X_i\le j^{1+\epsilon}\}}\ge Cj^{1+\epsilon}\right)\nonumber\\
        &\leq j\ \p(X_i\geq j^{1+\epsilon})+\frac{j\ \e\left(X_i\cdot1_{\{X_i\leq j^{1+\epsilon}\}}\right)}{Cj^{1+\epsilon}}\nonumber\\
        &\leq j^{-\epsilon}(1+C^{-1}(1+\epsilon)\log j).\label{eq:1.3}
    \end{align}
    Taking $j=2^m$  with $m\in\z^+$ in \eqref{eq:1.3}, combining with Borel-Cantelli lemma, we conclude that
    \begin{align*}
        \limsup_{j\to\infty}\left(j^{-1-\epsilon}\sum\limits_{i=1}^{j}X_i\right)\leq \limsup_{m\to\infty}\left((2^{m})^{-1-\epsilon}\sum\limits_{i=1}^{2^{m+1}}X_i\right)=0\quad\quad a.s.
    \end{align*}
    As a result, for any $\epsilon>0$,
    \begin{equation}\label{eq:1.4}
        \limsup_{j\to\infty}\frac{\log\left(\sum\limits_{i=1}^{j}X_i\right)}{\log j}\leq 1+\epsilon \quad\quad a.s.
    \end{equation}
    Note that since $X_i\geq 1$, $\liminf\limits_{j\to\infty}\frac{\log \left(\sum_{i=1}^{j}X_i\right)}{\log j}\ge1$. Take $\epsilon\rightarrow 0$ in \eqref{eq:1.4}, we have
    \begin{equation}\label{eq:1.5}
        \lim_{j\to\infty}\frac{\log\r_{\sigma_j}}{\log j}{=} \lim_{n\to\infty}\frac{\log\left(\sum\limits_{i=1}^{j}X_i\right)}{\log j}=1\quad\quad a.s.
    \end{equation}
    Combining \eqref{eq:1.5} with \eqref{LIL-range} gives
    \begin{equation}
        \lim_{j\rightarrow \infty}\frac{\log\sigma_j}{\log j}=\lim_{j\rightarrow \infty}\frac{\log\sigma_j}{\log\r_{\sigma_j}}\cdot\frac{\log\r_{\sigma_j}}{\log j}=2\quad\quad a.s.
    \end{equation}
\end{proof}
Let $\visitoneset{n}$ be the set of sites that have been visited exactly once by $(S_k)_{0\le k\le n}$. The following lemma is intuitive and quite straightforward, however, for completeness, we still provide a proof.
\begin{lemma}\label{lem1.3}
    For any $i\in \N$ and $\sigma_i\leq n\leq \tau_{i+1}-1$, we have
    \begin{equation*}
        \visitoneset{n}\subset \{ N_{\sigma_i}, M_{\sigma_i}\}.
    \end{equation*}
\end{lemma}
\begin{proof}
For any $s \in \visitoneset{n}$, let $k \in [0, n]$ be the unique time such that $S_k = s$. Let $j \in [1, i]$ be the index for which $\tau_j \le k \le \tau_{j+1} - 1$. Since $S_l \neq s$ for all $l \in [0, k-1]$, it follows that $S_k$ is not in the range of the preceding path, i.e.,
\begin{equation}
S_k \notin [N_{k-1}, M_{k-1}]. \label{eq:1}
\end{equation}

If $k > \sigma_j$, then $S_k \notin [N_{\sigma_j}, M_{\sigma_j}]$, which implies $\tau_{j+1} \le k$. Consequently, the interval $(S_k)_{\sigma_j \le k \le \tau_{j+1} - 1}$ is entirely contained within $[N_{\sigma_j}, M_{\sigma_j}]$, so the range remains unchanged. 
However, if $k > \sigma_j$, then by the definition of $\tau_{j+1}$ as the first time after $\sigma_j$ exiting $[N_{\sigma_j}, M_{\sigma_j}]$, we must have $k \ge \tau_{j+1}$. This contradicts the assumption $\tau_j \le k \le \tau_{j+1}-1$.

As a result, we have $\tau_j \leq k \leq \sigma_j$.
Now, suppose $S_{k-1} < S_k$. Then $S_k \geq N_{k-1}$, which, together with \eqref{eq:1}, implies $S_k \geq M_{k-1} + 1$. Since $n \geq \sigma_j$ and $S_{\sigma_j} = S_{\tau_j - 1} < S_k$, and given that $S_l \neq S_k$ for all $l$ with $\tau_j \le k < l \leq n$, it must be that $S_l < S_k$ for all $k < l \leq n$. This forces $S_k = M_n = M_{\sigma_i}$.
A symmetric argument, applied to the case $S_{k-1} > S_k$, shows that $S_k = N_{\sigma_i}$. This completes the proof of the lemma.
\end{proof}

\begin{proposition}\label{prop 1.3}
    Define 
    \begin{equation*}
        \curvevisitoneset{n} := \left\{ s : \exists \ k\in[1,n] \text{ such that } T_k = s, \text{ and } T_l \neq s \text{ for } 1 \leq l \leq n,\ l \neq k \right\}.
    \end{equation*}
    Let $\curvevisitonenumber{n}=\#\curvevisitoneset{n}$, then 
    \begin{equation*}
        \lim_{M\to\infty} \frac{\log\left(\p\left(\max\limits_{1\leq n\leq \sigma-1}\curvevisitonenumber{n}\geq M\right)\right)}{2\sqrt{M}}=-1.
    \end{equation*}
\end{proposition}
The proof of this proposition is divided into a lower bound (Proposition \ref{prop3.1}) and an upper bound (Proposition \ref{prop4.1}) in the next sections,
but defer the details to the next section and now proceed to the proof of our main theorem.
\begin{proof}[Proof of Theorem \ref{main result} assuming Proposition \ref{prop 1.3}]
    By Lemma \ref{lem1.3}, for $\sigma_i\leq n\leq \tau_{i+1}-1$, $\visitonenumber{n}\leq 2$.
    For $\tau_i\leq n\leq \sigma_i-1$, we have
    \begin{equation*}
        \visitoneset{n}\subset \left\{N_{\sigma_{i-1}},M_{\sigma_{i-1}}\right\}\cup \mathcal{E}_i(n),
    \end{equation*}
    where $\mathcal{E}_i(n):=\left\{s: T_{k-\tau_i+1}^i=s\text{~for~some~} k\in[\tau_i, n] \text{ and } T_{l-\tau_i+1}^i \neq s \text{ for } \tau_i \leq l \leq n,\ l \neq k \right\}$ with $i\ge1$. Denote $d_i(n)=\# \mathcal{E}_i(n)$.
    Then
    \begin{equation}\label{eq:1.8}
        \max\limits_{\tau_i\leq n\leq \sigma_i-1}d_i(n)\leq \max\limits_{\tau_i\leq n\leq \sigma_i-1}\visitonenumber{n}\leq \max\limits_{\tau_i\leq n\leq \sigma_i-1}d_i(n)+2.
    \end{equation}
    Notice that $D_i:=\max\limits_{\tau_i\leq n\leq \sigma_i-1}d_i(n)$ are i.i.d.~random variables with the same distribution as $\max\limits_{1\leq n\leq \sigma-1}\curvevisitonenumber{n}$. Therefore, by Lemma \ref{lem1.2},
    \begin{equation}\label{eq:1.9}   
        \limsup_{n\to\infty}{\frac{\visitonenumber{n}}{\log^2 n}}=\limsup_{i\to\infty} \frac{\max\limits_{\tau_i\leq n\leq \sigma_i-1}d_i(n)}{\log^2\sigma_i}=\frac{1}{4}\limsup_{i\to\infty} \frac{\max\limits_{\tau_i\leq n\leq \sigma_i-1}d_i(n)}{\log^2 i}.
    \end{equation}
    By Proposition \ref{prop 1.3}, for any $C>0$,
    \begin{equation}\label{eq:1.10}
        \p\left(\max\limits_{\tau_i\leq n\leq \sigma_i-1}d_i(n)\geq C\log^2 i\right)=\exp\left({-(1+o(1))2\sqrt{C}\log i}\right).
    \end{equation}
The sum of the series converges if $2\sqrt{C} > 1$ and diverges if $2\sqrt{C} < 1$. By the Borel-Cantelli lemma, the threshold is therefore $C = \frac14$.
Then plugging \eqref{eq:1.10} and the Borel-Cantelli lemma into \eqref{eq:1.9} gives Theorem \ref{main result}.
\end{proof}

\section{Proof of Proposition \ref{prop 1.3}}\label{tech-section}
The rest of the paper is devoted to proving Proposition \ref{prop 1.3}. We begin in Section \ref{Tech-esti} by establishing the required technical estimates. The main steps of the proof of the lower and upper bounds are then presented in Sections \ref{se:3} and \ref{se:4}, respectively.

\subsection{Technical Estimates}\label{Tech-esti}
Let $(G_n)_{n\ge 1}$ be a SRW starting at $G_1=1$, conditioned never to hit $0$. 
That is, $(G_n)$ is a $h-$process.  In other words, starting from $G_1=1$ and with respect to its natural filtration $\mathcal{F}^{G}_{n}$, its transition probabilities are explicitly given as follows:
\begin{align*}
    &\p\left(G_{n+1}=G_n+1\mid \mathcal{F}^{G}_{n}\right)=\frac{G_n+1}{2G_n},\\
    &\p\left(G_{n+1}=G_n-1\mid \mathcal{F}^{G}_{n}\right)=\frac{G_n-1}{2G_n}.
\end{align*}
Define $\tau^G_s:=\inf\{n:G_n=s\}$ as the stopping time when  $(G_n)$ hits $s$ for the first time. Define
\begin{equation*}
    \mathcal{G}_s:=\left\{ x : \exists \ k\in[1, \tau^G_s] \text{ such that } G_k = x, \text{ and } G_l \neq x \text{ for } 1 \leq l \leq \tau^G_s,\ l \neq k \right\}
\end{equation*}
which is the set visited once at time $\tau^G_s$. Let $g(s)=\#\mathcal{G}_s$.
\begin{proposition}\label{prop1.5}
For any $\epsilon>0$, $0<c_1<c_2$, there exists $N=N(\epsilon,c_1,c_2)>0$ such that for any $s\geq N$, $C\in [c_1,c_2]$, we have
    \begin{equation*}
       \exp\left({-(2+\epsilon)C\log s}\right)\leq \p\left(g(s)\geq C\log^2 s\right)\leq \exp\left({-(2-\epsilon)C\log s}\right).
    \end{equation*}
\end{proposition}
By adapting ideas from \cite[Theorem 2]{Ma88}, Proposition \ref{prop1.5} can be proved by verifying the following two lemmas. We remark that the first lemma is similar to \cite[Lemma 4]{Ma88}.
\begin{lemma}\label{lem1.6}
For any $0<\epsilon<1$, $0<c_1<c_2$, there exists $N=N(\epsilon,c_1,c_2)>0$ such that for any $s\geq N$, $k\in [c_1\log s,c_2\log s]$, we have
\begin{equation*}
    \left(\frac{(1-\epsilon)\log s}{2}\right)^k\leq \e \left[\binom{g(s)}{k}\right]\leq \left(\frac{(1+\epsilon)\log s}{2}\right)^k.
\end{equation*}
\end{lemma}

\begin{lemma}\label{lem1.7}
    If a sequence of random variables $X_s\in \N$ satisfy that for any $0<\epsilon<1$, $0<c_1<c_2$, there exists $N=N(\epsilon,c_1,c_2)>0$ such that for any $s\geq N$, $k\in [c_1\log s,c_2\log s]$, 
\begin{equation*}
    \left(\frac{(1-\epsilon)\log s}{2}\right)^k\leq \e \left[\binom{X_s}{k}\right]\leq \left(\frac{(1+\epsilon)\log s}{2}\right)^k,
\end{equation*}
then for any constants $\epsilon>0$, $0<c_1<c_2$, there exists $N'=N'(\epsilon,c_1,c_2)>0$ such that for any $s\geq N'$, $C\in [c_1,c_2]$, it holds that
    \begin{equation*}
       \exp\left(-(2+\epsilon)C\log s\right)\leq \p\left(X_s\geq C\log^2 s\right)\leq \exp\left(-(2-\epsilon)C\log s\right).
    \end{equation*}
\end{lemma}

\begin{proof}[Proof of Lemma \ref{lem1.6}]
    For any set $A=\{x_1,x_2,\dots ,x_k\}\subset \{1,2,\dots ,s-1\}$, suppose $x_1<x_2<\dots<x_k$ and $k\le s-1$. 
    Then $A\subset \mathcal{G}_s$ if and only if $G_{n_1}>x_i$ for any $1\leq i\leq k-1$ and $\tau^G_{x_i}+1\leq {n_1}\leq \tau^G_{x_{i+1}}$, and $G_{n_2}>x_k$ for any $\tau^G_{x_k}+1\leq n_2\leq \tau^G_{s}$. 
    Similar to our previous techniques \cite[(4.7)]{FH25}, we see that $(\frac{1}{G_n})_{n\ge1}$ is martingale. By applying the Optional Stopping theorem at $\tau^G_{a}\wedge\tau^G_{b}$, with the initial state $\frac{1}{G_1}=\frac{1}{a+1}$. One has,
    \begin{equation}\label{mart-anal}
        \p\left(G_n>a,\ \mbox{ for any } \tau^G_{a}+1\leq n\leq \tau^G_{b}\,\Big|\, \mathcal{F}^{G}_{\tau^G_{a}}\right)=\frac{b}{2a(b-a)}.
    \end{equation}
    Then by Markov property,
    \begin{align}
        \p(A\subset \mathcal{G}_s)=\,\;&\prod_{i=1}^{k-1}\p\left(G_y>x_i, \mbox{ for any }  \tau^G_{x_i}+1\leq y\leq \tau^G_{x_{i+1}}\,\Big|\, \mathcal{F}^{G}_{\tau^G_{x_i}}\right)\notag\\
        &\cdot \p\left(G_y>x_k, \mbox{ for any } \tau^G_{x_k}+1\leq y\leq \tau^G_s\,\Big|\, \mathcal{F}^{G}_{x_k}\right)\notag\\
        =\,\;&\prod_{i=1}^{k-1}\frac{x_{i+1}}{2x_i(x_{i+1}-x_i)}\cdot \frac{s}{2x_k(s-x_k)}=2^{-k}\prod_{i=1}^{k-1}\frac{1}{x_{i+1}-x_i}\cdot \frac{s}{x_1(s-x_k)}.\label{eq:1.11}
    \end{align}
    Note that $s\in \mathcal{G}_s$. Let A by a $k$-element subsets of $\{1,2\dots s-1\}$, Take $a_1=x_1,a_2=x_2-x_1,...,a_k=x_k-x_{k-1},a_{k+1}=s-x_k$, summing up \eqref{eq:1.11} gives
    \begin{align*}
        \e\left[\binom{g(s)-1}{k}\right]=\,\;&2^{-k}\sum_{1\leq x_1<x_2<\dots<x_{k}\leq s-1}\prod_{i=1}^{k-1}\frac{1}{x_{i+1}-x_i}\cdot \frac{s}{x_1(x-x_k)}\\
        =\,\;&2^{-k}\sum_{\substack{a_1+a_2+\dots+a_{k+1}=s\\ 1\le a_i\le s,\ \forall i\in[1,k+1]}}\frac{a_1+a_2+\dots+a_{k+1}}{a_1a_2\dots a_{k+1}}\\
        =\,\;&2^{-k}\sum_{j=1}^{k+1}\sum_{\substack{a_1+a_2+\dots+a_{k+1}=s\\ 1\le a_i\le s,\ \forall i\in[1,k+1]}}\dfrac{1}{\prod\limits_{i\neq j}a_i}\\
        =\,\;&2^{-k}(k+1)\sum_{\substack{a_1+a_2+\dots+a_{k}\leq s-1\\ 1\le a_i\le s,\ \forall i\in[1,k]}}\frac{1}{a_1 a_2\dots a_{k}}.
    \end{align*}
    Note that 
    \begin{align*}
        &\sum_{a_1+a_2+\dots+a_{k}\leq s-1}\frac{1}{a_1 a_2\dots a_{k}}\leq \left(\sum_{a\leq s-1}\frac{1}{a}\right)^k\leq (1+\log s)^k;\\
        &\sum_{a_1+a_2+\dots+a_{k}\leq s-1}\frac{1}{a_1 a_2\dots a_{k}}\geq \left(\sum_{a\leq \frac{s-1}{k}}\frac{1}{a}\right)^k\geq (\log s-\log k)^k,
    \end{align*}
    one has
    \begin{align*}
        &\e\left[\binom{g(s)-1}{k}\right]\leq 2^{-k}(k+1)(1+\log s)^{k};\\
        &\e\left[\binom{g(s)-1}{k}\right]\geq 2^{-k}(k+1)(\log s-\log k)^{k}.
    \end{align*}
    Then for $k\in [c_1\log s,c_2\log s]$, we have
    \begin{align*}
        &\left(\frac{(1-\epsilon)\log s}{2}\right)^k\leq \e\left[\binom{g(s)}{k}\right]
        =\e\left[\binom{g(s)-1}{k}\right]+\e\left[\binom{g(s)-1}{k-1}\right]\leq \left(\frac{(1+\epsilon)\log s}{2}\right)^k.
    \end{align*}
\end{proof}
Having completed the proof of Lemma \ref{lem1.6}, we now turn to Lemma \ref{lem1.7}.
\begin{proof}[Proof of Lemma \ref{lem1.7}]
    Take $k_1=\lfloor 2C\log s\rfloor$ and $\epsilon_0$ is a small constant. By Markov's inequality,
    \begin{align}
        \p(X_s\geq C\log^2 s)&\leq \binom{\lceil C\log^2 s\rceil}{k_1}^{-1}\e \left( \binom{X_s}{k_1}\right)\notag\\
        &\leq \binom{\lceil C\log^2 s\rceil}{k_1}^{-1}\left(\frac{(1+\epsilon_0)\log s}{2}\right)^{k_1}.\label{eq:1.12}
    \end{align}
    By Stirling's formula,
    \begin{align*}
        \text{RHS~of~\eqref{eq:1.12}}\leq o(\log s)\left(\frac{1+\epsilon_0}{\text{e}}\right)^{2C\log s}\leq e^{-(2-\epsilon)C\log s}
    \end{align*}
    for $\epsilon_0<\epsilon$.
    
    For the lower bound, note that for any $0<\epsilon_0<\frac{\min\{1,C\}}{10}$, taking $\epsilon_1<\frac{\epsilon_0^2}{12C(C+\epsilon_0)}$, for sufficiently large $s$, we have
    \begin{align}\label{LDP01}
        &\e \left[ \binom{X_s}{k_1}1_{\{X_s\geq (C+\epsilon_0)\log^2 s\}}\right]
        \le\sum_{r\ge (C+\epsilon_0)\log^2 s} \binom{r}{k_1}\cdot\p(X_s=r)\nonumber\\
        \le\,\;&\binom{\lceil(C+\epsilon_0)\log^2 s\rceil}{\lfloor 2C\log s\rfloor}\binom{\lceil(C+\epsilon_0)\log^2 s\rceil}{\lfloor 2(C+\epsilon_0)\log s\rfloor}^{-1}\sum_{r\ge (C+\epsilon_0)\log^2 s}\binom{r}{\lfloor2(C+\epsilon_0)\log s\rfloor}\p(X_s=r)\nonumber\\
        \leq\,\;& \exp{\left(-2\epsilon_0\log s\cdot\log\left(\frac{\log s}{2}\right)-\frac{\epsilon_0^2\log s}{2(C+\epsilon_0)}+O(\log\log s)\right)}\cdot\e \left( \binom{X_s}{\lfloor2(C+\epsilon_0)\log s\rfloor}\right)\nonumber\\
        \leq\,\;& \exp \left(6C\epsilon_1\log s- \frac{\epsilon_0^2\log s}{2(C+\epsilon_0)}+O(\log\log s)\right)\left(\frac{(1-\epsilon_1)\log s}{2}\right)^{k_1}\nonumber\\
        =&\,\;o(1)\left(\frac{(1-\epsilon_1)\log s}{2}\right)^{k_1}.
    \end{align}
    And similarly,
    \begin{align}\label{LDP02}
        &\e \left[ \binom{X_s}{k_1}1_{\{X_s\leq (C-\epsilon_0)\log^2 s\}}\right]
        =o(1)\left(\frac{(1-\epsilon_1)\log s}{2}\right)^{k_1}.
    \end{align}
    Then combining \eqref{LDP01} and \eqref{LDP02}, using Stirling's formula again, we get
    \begin{align*}
        &\p\left(X_s\geq (C-\epsilon_0)\log^2 s\right)
        \ge\sum_{r=\lceil (C-\epsilon_0)\log^2 s\rceil}^{\lfloor (C+\epsilon_0)\log^2 s\rfloor}\frac{\binom{r}{k_1}}{\binom{\lfloor (C+\epsilon_0)\log^2 s\rfloor}{k_1}}\cdot\p\left(X_s=r\right)\\
        =\,\;&\binom{\lfloor (C+\epsilon_0)\log^2 s\rfloor}{k_1}^{-1}\cdot\e \left[ \binom{X_s}{k_1}1_{\{(C-\epsilon_0)\log^2 s\leq X_s \leq (C+\epsilon_0)\log^2 s\}}\right]\\
        \geq\,\; &(1-o(1))\binom{\lfloor (C+\epsilon_0)\log^2 s\rfloor}{k_1}^{-1}\left(\frac{(1-\epsilon_1)\log s}{2}\right)^{k_1}\\
        \geq\,\;&\exp(-2C\log s-2\epsilon_0\log s-4C\epsilon_1\log s)\geq \exp\left({-(2+\epsilon)C\log s}\right)
    \end{align*}
    for $\epsilon_0$ small enough. 
\end{proof}

\subsection{The Lower Bound}\label{se:3}
We now prove the lower bound in Proposition \ref{prop 1.3}. Specifically,
\begin{proposition}\label{prop3.1}
\begin{equation}\label{lower-bound}
    \liminf_{M\to\infty} \frac{\log\left(\p\left(\max\limits_{1\leq m\leq \sigma-1}\curvevisitonenumber{m}\geq M\right)\right)}{2\sqrt{M}}\geq -1.
\end{equation}  
\end{proposition}

The core idea of this subsection is a multi-stage construction.
It is worth noting that the proof is built with a bootstrapping, iterative framework designed for analyzing the maximum number of once visited sites at each excursion. 
We also remark that although the desired lower bound is in principle also achievable via describing a non-iterative explicit tilting strategy. Such approach is significantly more intricate than ours.

To prepare for the proof, we begin by introducing the following notation.
Let
\begin{equation*}
    \tilde{\mathcal{G}}_s:=\left\{ y \leq \frac{s}{2}: \exists \ k,\ 1 \leq k \leq \tau^G_s, \text{ such that } G_k = y, \text{ and } G_l \neq y \text{ for } 1 \leq l \leq \tau^G_s,\ l \neq k \right\}
\end{equation*}
stand for the set of once visited sites of $G$ smaller than the site $\frac{s}{2}$ at time $\tau^G_s$. Let $\tilde g(s)=\#\tilde{\mathcal{G}}_s$
and
\begin{equation}\label{Lambda}
\Lambda:=\liminf_{M\to\infty} \frac{\log(\p(\max\limits_{1\leq m\leq \sigma-1}\curvevisitonenumber{m}\geq M))}{2\sqrt{M}}.
\end{equation}
We will first prove that the behavior of $\tilde{\mathcal{G}}_s$ is similar to $\mathcal{G}_s$ with the following lemma.
\begin{lemma}\label{lem3.2}
For any $0<\epsilon<1$, $0<c_1<c_2$, there exists $N=N(\epsilon,c_1,c_2)>0$ such that for any $s\geq N$, $k\in [c_1\log s,c_2\log s]$, one has
\begin{equation*}
    \left(\frac{(1-\epsilon)\log s}{2}\right)^k\leq \e \left( \binom{\tilde{g}(s)}{k}\right)\leq \left(\frac{(1+\epsilon)\log s}{2}\right)^k.
\end{equation*}
\end{lemma}
This lemma is exactly the same as Lemma \ref{lem1.6}, except that $g(s)$ is replaced by $\tilde{g}(s)$. The proof is similar as well.
\begin{proof}
     The upper bound can be verified immediately from Lemma \ref{lem1.6} since $\tilde{g}(s)\leq g(s)$. For the lower bound, 
     running over all $k$-element subsets of $\{1,2\dots \lfloor \frac{s}{2}\rfloor\}$, and summing up \eqref{eq:1.11} gives
     \begin{align*}
          \e\left(\binom{\tilde{g}(s)}{k}\right)=&2^{-k}\sum_{1\leq x_1<x_2<\dots<x_{k}\leq \lfloor \frac{s}{2}\rfloor}\prod_{i=1}^{k-1}\frac{1}{x_{i+1}-x_i}\cdot \frac{s}{x_1(s-x_k)}\\
          \geq&2^{-k} \sum_{1\leq x_1<x_2<\dots<x_{k}\leq \lfloor \frac{s}{2}\rfloor}\prod_{i=1}^{k-1}\frac{1}{x_{i+1}-x_i}\cdot \frac{1}{x_1}\\
         =&2^{-k}\sum_{a_1+a_2+\dots +a_k\leq \lfloor \frac{s}{2}\rfloor}\frac{1}{a_1a_2\dots a_k}\geq 2^{-k}\left(\log s-\log(2k)\right)^k\\
         \geq &\left(\frac{(1-\epsilon)\log s}{2}\right)^k
     \end{align*}
     for $k\in [c_1\log s,c_2\log s]$.
\end{proof}
\begin{proof}[Proof of Proposition \ref{prop3.1}]
Combining Lemma \ref{lem3.2} and Lemma \ref{lem1.7}, for any $\epsilon>0$, $0<c_1<c_2$, there exists $N=N(\epsilon,c_1,c_2)>0$ such that for any $s\geq N$, $C\in [c_1,c_2]$, we have
    \begin{equation}\label{eq:3.0}
       \exp\left({-(2+\epsilon)C\log s}\right)\leq \p\left(\tilde{g}(s)\geq C\log^2 s\right)\leq \exp\left({-(2-\epsilon)C\log s}\right).
    \end{equation}
Choose $0<\epsilon_0<0.1$ and let $K_M=\lfloor \text{e}^{2(-\Lambda-\epsilon_0)\sqrt{M}}\rfloor$. Define the stopping times
\begin{align*}
    &\beta_0=\gamma_0=\inf\{n:T_n=K_M\};\\
    &\beta_i=\inf\{n>\gamma_{i-1}:T_n=\max_{1\leq k\leq \gamma_{i-1}}T_k+1\};\\
    &\gamma_i=\inf\{n>\gamma_{i-1}:T_n=T_{\beta_i}-1\};\\
    &\beta_{\rm fail}=\inf\{n>\beta_0:T_n\leq\frac{K_M}{2}\};\\
    &\beta_{\rm succ}=\inf\{n>\beta_0:T_n= 2K_M\}.
\end{align*}
Furthermore, the process $(T_{n+\beta_i-1} - T_{\beta_i-1})_{1 \le n \le \gamma_i - \beta_i + 1}$ is a simple random walk starting at $1$ and stopping upon hitting $0$. This implies that $(T_{n+\beta_i-1})_{0 \le n \le \gamma_i - \beta_i + 1}$ has the same distribution as $(T_n)_{0 \le n \le \sigma}$.

Now, define $Y_i = \max\limits_{1 \le n \le \gamma_i} T_n - \max\limits_{1 \le n \le \gamma_{i-1}} T_n$. Then, on the event ${\beta_i < \infty}$, the distribution of $Y_i$ is identical to that of $X_i$ in \eqref{eq:logvariable} and is independent of $\mathcal{F}_{\beta_i}^T$.
Let $l = \lfloor K_M / \log^2 K_M \rfloor$, and define the events $A_1$ and $A_2$ by
\begin{align*}
    A_1:=\{\beta_{\rm succ}<\beta_{\rm fail}\},\quad\mbox{and}\quad
    A_2:=\{\gamma_{l}< \beta_{\rm succ}\}.
\end{align*}
Then on $\beta_0<\infty$,
\begin{align*}
    &\p(A_1\mid \mathcal{F}^T_{\beta_0})\geq \frac{1}{3};\\
    &\p(A_2^c\mid \mathcal{F}^T_{\beta_0})=\p\left(\sum_{i=1}^{l}X_i\ge K_M\right)\leq l\cdot\p(X_i\geq K_M)+\frac{l\cdot\e(X_i\cdot1_{X_i<K_M})}{K_M}=O(\log^{-1}(K_M)).
\end{align*}
Define $Z_i$ as the maximum number of sites that are visited exactly once by the sequence $(T_k)_{1\le k\le \gamma_i-\beta_i+1}$. That is,
\begin{align*}
    &Z_i=\max_{1\leq m\leq \gamma_i-\beta_i+1}\\
    &\#\left\{s:\exists\ k\in[1,m] \text{ such that } T_{k+\beta_i-1} = s, \text{ and } T_{j+\beta_i-1} \neq s \text{ for } 1 \leq j \leq m,\ j \neq k \right\}.
\end{align*}
Then, $\{Z_i\}_{i\ge1}$ is an i.i.d.~sequence, with each $Z_i$ distributed as $\max\limits_{1\leq m\leq \sigma-1}d(m)$. 
Hence, for any constants $0<c<1$, and $0<\epsilon_1<0.1$, by \eqref{Lambda}, for sufficiently large $M$,
\begin{equation*}
    \p(Z_i\geq cM)\geq \exp\left({-2(-\Lambda+\epsilon_1)\sqrt{cM}}\right).
\end{equation*}
Take $c_3=1+\frac{2(\epsilon_0+\epsilon_1)}{\Lambda}$ and recall that $0<\epsilon_0<\frac{\min\{1,C\}}{10}$, then
\begin{equation*}
    \p(Z_i\geq c_3 M)\geq \exp\left({-2\left(-\Lambda-\epsilon_0+\frac{\epsilon_1^2}{\Lambda}\right)\sqrt{M}}\right)>K_M^{-1+\frac{\epsilon_1^2}{\Lambda^2}}.
\end{equation*}
Since $\{Z_i\}_{i\ge1}$ are i.i.d.~random variables, $\p(\max_{1\leq i\leq l}Z_i\geq cM)=1-o(1)$. Combining the above estimates, we obtain
\begin{equation}\label{eq:3.2}
    \p\left(\ \exists\ i\geq 1,\ \gamma_i<\beta_{\rm fail},\ Z_i\geq c_3M\ \big|\ \mathcal{F}_{\beta_0}^T\right)\geq \p(A_i)-\p(A_2^c)-o(1)\geq \frac{1}{4}.
\end{equation}
Note that 
\begin{equation}\label{eq:3.3}
    \p(\beta_0<\infty)=K_M^{-1}.
\end{equation}
Conditioned on $\beta_0<\infty$, the process $(T_n)_{1\leq n\leq \beta_0}$ is a simple random walk starting at $1$, which is conditioned never to hit $0$ and stopped upon hitting $K_M$. Denote
\begin{equation*}
    \tilde{\mathcal{D}}_{K_M}:=\left\{1\leq s\leq \frac{K_M}{2}:\exists \ 1\leq k\leq \beta_0, \text{ such that } T_{k} = s, \text{ and } T_{l} \neq s \text{ for } 1 \leq l \leq \beta_0,\ l \neq k \right\}.
\end{equation*}
By \eqref{eq:3.0}, for any $\epsilon>0$,
\begin{equation}\label{eq:3.4}
    \p\left(\left.\#\tilde{\mathcal{D}}_{K_M}\geq \left(1-c_3\right)M\right|\beta_0<\infty \right)
    =\p\left(\tilde{g}(K_M)\geq \left(1-c_3\right)M\right) \geq\text{e}^{-(2+\epsilon)\left(\frac{\epsilon_0+\epsilon_1}{\Lambda(\Lambda+\epsilon_0)}\right)\sqrt{M}}.
\end{equation}
For $\beta_i\leq n\leq \gamma_i$ satisfies $\gamma_i<\beta_{\rm fail}$, define
\begin{equation*}
    \widehat{\mathcal{D}}_{n}:=\left\{x:\exists\ k,\ \beta_i\leq k\leq n \text{ such that } T_{k} = x, \text{ and } T_{l} \neq x \text{ for } \beta_i \leq l \leq n,\ l \neq k \right\}.
\end{equation*}
Then for $\beta_i\leq n_1\leq \gamma_i$ satisfies $\gamma_i<\beta_{\rm fail}$, 
\begin{equation}\label{eq:3.5}
    \tilde{\mathcal{D}}_{K_M}\cup  \widehat{\mathcal{D}}_{n_1}\subset\mathcal{D}_{n_1}.
\end{equation}
Combining \eqref{eq:3.2}, \eqref{eq:3.3}, \eqref{eq:3.4} and \eqref{eq:3.5} gives
\begin{align*}
    &\p\left(\max_{1\leq m\leq \sigma-1}\curvevisitonenumber{m}\geq M\right)\\
    \geq& \;\,\p(\beta_0<\infty)\cdot\p\left(\left.\#\tilde{\mathcal{D}}_{K_M}\geq \left(1-c_3\right)M\right|\beta_0<\infty \right)
    \cdot\p\left(\left.\exists\ i\geq 1,\ \gamma_i<\beta_{\rm fail},\ Z_i\geq c_3 M\right|\mathcal{F}_{\beta_0}\right)\\
    \geq& \;\,\exp\left({-\left(-2(\Lambda+\epsilon_0)+(2+\epsilon)\left(\frac{\epsilon_0+\epsilon_1}{\Lambda(\Lambda+\epsilon_0)}\right)\right)\sqrt{M}+O(1)}\right).
\end{align*}
Taking $\epsilon_1$ and $\epsilon$ close to $0$,
we obtain
\begin{align*}
    \Lambda\geq \Lambda+\epsilon_0-\frac{\epsilon_0}{\Lambda(\Lambda+\epsilon_0)}
\end{align*}
for any $\epsilon_0>0$, which gives $\Lambda\geq -1$ and completes the proof of \eqref{lower-bound}.
\end{proof}

\subsection{The Upper Bound}\label{se:4}
In this subsection, we prove the upper bound of Proposition \ref{prop 1.3}, namely:
\begin{proposition}\label{prop4.1}
\begin{equation*}
    \limsup_{M\to\infty} \frac{\log\left(\p\left(\max\limits_{1\leq m\leq \sigma-1}\curvevisitonenumber{m}\geq M\right)\right)}{2\sqrt{M}}= -1.
\end{equation*}  
\end{proposition}
The upper bound is established through a delicate bootstrapping (self-improving) iterative framework.
The proof relies on three key lemmas (Lemmas \ref{lem4.2}–\ref{lem4.4}). Specifically, for the upper bound, we first observe that once-visited sites that are spatially close exhibit strong correlations. To avoid overestimation inherent in a direct union bound, we introduce a ``cross-graining'' technique. Subsequently, by partitioning the process via stopping times, we show that the maximum of once-visited sites within each segment inherits a self-similar structure relative to the global maximum. This self-similarity further necessitates a self-boosting argument to complete the proof.

Within this framework, the detailed proofs of the three lemmas are deferred, as they are technical yet essential for establishing Proposition \ref{prop4.1}. For any $c > 0$, define
\begin{equation}
    h(c):=\limsup_{M\to\infty} \frac{\log\left(\p\left(\max\limits_{1\leq s\leq \text{e}^{c\sqrt{M}}}g(s)\geq M\right)\right)}{2\sqrt{M}}.\label{eq:4.0}
\end{equation}
\begin{lemma}\label{lem4.2}
    For any $0<N_1<N_2$ such that $\p\left(\max\limits_{1\leq s\leq N_2}g(s)\geq M\right)\leq 0.1$, we have
    \begin{equation}
        \p\left(\max\limits_{1\leq s\leq N_2}g(s)\geq M\right)\geq \frac{N_2-N_1}{8N_1(\log N_1+1)}\cdot\p\left(\max\limits_{1\leq s\leq N_1}g(s)\geq M\right).\label{eq:lem4.2}
    \end{equation}
\end{lemma}
Suppose 
\begin{equation}\label{eq:3.1}
     \Theta:=\limsup_{M\to\infty} \frac{\log\left(\p\left(\max\limits_{1\leq m\leq \sigma-1}\curvevisitonenumber{m}\geq M\right)\right)}{2\sqrt{M}}.
\end{equation}
\begin{lemma}\label{lem4.3}
    For any $0<c<-2\Theta+\frac{(1+\Theta)^{2}}{2}$,
    \footnote{The motivation for writing the expression as $-2\Theta+\frac{(1+\Theta)^2}{2}$ is to clearly show that it exceeds $-2\Theta$ by precisely $\frac{(1+\Theta)^2}{2}$. The same logic applies to the analogous expressions that follow.}
    \begin{equation}
        h(c)\leq \frac{c}{2}+\Theta-\frac{(1+\Theta)^2}{4}.\label{eq:lem4.3}
    \end{equation}
\end{lemma}
We now proceed to prove Proposition \ref{prop4.1} on the basis of the two lemmas established above.

\begin{proof}[Proof of Proposition \ref{prop4.1}]
 Note that by Lemma \ref{lem4.3}, $h(-2\Theta+\frac{(1+\Theta)^2}{4})\leq -\frac{(1+\Theta)^2}{8}$. Denote $N'=\lfloor\text{e}^{(-2\Theta+\frac{(1+\Theta)^2}{4})\sqrt{M}}\rfloor$, then by Lemma \ref{lem4.2},
\begin{eqnarray*}
    &&\p\left(\max\limits_{1\leq m\leq \sigma-1}d(m)\geq M\right)\\
    &\leq& \sum_{i=1}^{N'}\p\left(\max\limits_{1\leq m\leq \sigma-1}d(m)\geq M\,\Big|\, \max\limits_{1\leq n\leq \sigma-1}T_n=i\right)\cdot\p\left(\max\limits_{1\leq n\leq \sigma-1}T_n=i\right)+\p\left(\max\limits_{1\leq n\leq \sigma-1}T_n>N'\right)\\
    &\overset{\eqref{T-distribution}}{\leq} &\sum_{i=1}^{N'}\frac{\p\left(\max\limits_{1\leq s\leq i}g(s)\geq M\right)}{i(i+1)}+\frac{1}{N'+1}\\
    &\overset{\eqref{eq:lem4.2}}{\leq}  &\sum_{i=1}^{N'}\frac{8(\log(N')+1)\cdot\p\left(\max\limits_{1\leq s\leq N'}g(s)\geq M\right)}{(N'-i)(i+1)}+\frac{1}{N'+1}\\
    &\leq &\frac{1}{N'}\cdot\p\left(\max\limits_{1\leq s\leq N'}g(s)\geq M\right)\cdot(2\log N'+1)^2+\frac{1}{N'}\leq \frac{1}{N'}(2\log(N')+3)^2.
\end{eqnarray*}
Consequently,
\begin{align*}
     2\Theta\sqrt{M}\leq \left(2\Theta-\frac{(1+\Theta)^2}{4}\right)\sqrt{M}
\end{align*}
which implies $\Theta=-1$, then we get Proposition \ref{prop4.1}.
\end{proof}

We now turn to the proofs of the two lemmas presented at the beginning of this subsection.
\begin{proof}[Proof of Lemma \ref{lem4.2}]
    Define $M_n^G:=\max_{1\leq k\leq n}G_k$. Define $\beta_1'=1$, for $i\geq 1$, define stopping times
    \begin{align*}
        &\gamma_i':=\inf\left\{n>\beta_{i}',\ G_n\in \{G_{\beta_i'}-1,G_{\beta_i'}+N_1-1\}\right\};\\    &\beta_{i+1}':=\inf\{n>\gamma_i',\ G_n=M^G_{\gamma_i'}+1\}.
    \end{align*}
    For $(G_n)$, an interval $[\beta_i',\gamma_i']$ is called a good interval if and only if $G_{\gamma_i'}=G_{\beta_i'}+N_1-1$. Otherwise, we call it a bad interval. Let $\eta_i$ be the end time of the $i$-th good interval. Similarly to the approach in \eqref{mart-anal},
    \begin{equation*}           \p\left([\beta_i',\gamma_i']\text{~is~a~good~interval}\,\Big|\,\mathcal{F}_{\beta_i'}^G\right)=\frac{(G_{\beta_i'}+N_1-1)}{N_1G_{\beta_i'}}
    \end{equation*}
    and
    \begin{equation*}
        \e\left(M^G_{\gamma_i'}-M^G_{\beta_i'}\,\Big|\,\mathcal{F}_{\beta_i'}^G\right)=\sum_{j=1}^{N_1}\frac{(G_{\beta_i'}+j-1)}{jG_{\beta_i'}}\leq N_1\sum_{j=1}^{N_1}\frac{1}{j}\p\left([\beta_i',\gamma_i']\text{~is~a~good~interval}\,\Big|\,\mathcal{F}_{\beta_i'}^G\right),
    \end{equation*}
    which gives $\e(M_{\eta_{i+1}}^G-M_{\eta_i}^G\,\big|\, \mathcal{F}^G_{\eta_i})\leq N_1\sum\limits_{j=1}^{N_1}\frac{1}{j}$. Since $M^G_{\gamma_1'}=N_1$. By Markov's inequality,
    \begin{equation}
        \p\left(M^G_{\eta_{\lfloor \frac{N_2-N_1}{2N_1(\log N_1+1)}\rfloor+1}}\geq N_2\right)\leq \frac{1}{2}\label{eq:4.1}.
    \end{equation}
    Conditioned on $[\beta_i',\gamma_i']$ being a good interval, $(G_k)_{\beta_i'\le k\le\gamma_i'}$ is a SRW starting at $G_{\beta_i'}$, conditioned on hitting $G_{\beta_i}+N_1-1$ before hitting $G_{\beta_i}-1$, stopping at hitting $G_{\beta_i'}+N_1-1$, independent with $\mathcal{F}^G_{\beta_i'}$. Therefore, $(G_{n+\beta_i'-1}-G_{\beta_i'-1})_{1\leq n\leq \gamma_i'-\beta_i'+1}$ has the same distribution as $(G_n)_{1\leq n\leq \tau^G_{N_1}}$, which implies
    \begin{equation}
        \p\left(\max_{G_{\beta_i'}\leq s\leq G_{\gamma_i'}}g(s)\geq M\,\Big|\, \mathcal{F}^G_{\beta_i'},\ [\beta_i',\gamma_i']\text{~is~a~good~interval}\right)\geq  \p\left(\max_{1\leq s\leq N_1}g(s)\geq M\right).\label{eq:4.2}
    \end{equation}
    Combining \eqref{eq:4.1} and \eqref{eq:4.2} gives
    \begin{equation}
        \p\left(\max_{1\leq s\leq N_2}g(s)\geq M\right)\geq \frac{1}{2}\left(1-\left(1- \p\left(\max_{1\leq s\leq N_1}g(s)\geq M\right)\right)^{\lfloor\frac{N_2-N_1}{2N_1(\log N_1+1)}\rfloor+1}\right).\label{eq:4.3}
    \end{equation}
    Then we have
    \begin{align*}
         \p\left(\max\limits_{1\leq s\leq N_2}g(s)\geq M\right)\geq &-\frac{1}{4}\log\left(1-2\p\left(\max\limits_{1\leq s\leq N_2}g(s)\geq M\right)\right)\\
         \geq& \frac{N_2-N_1}{8{N_1}(\log(N_1)+1)}\cdot \p\left(\max_{1\leq s\leq N_1}g(s)\geq M\right).
    \end{align*}
\end{proof}

An immediate corollary of Lemma \ref{lem4.2} is the following.
Taking $N_1=\lfloor \text{e}^{c_1\sqrt{M}}\rfloor$ and $N_2=\lfloor \text{e}^{c_2\sqrt{M}}\rfloor$ in Lemma \ref{lem4.2}, if \eqref{eq:lem4.3} holds for some constant $c_2\in (0,-2\Theta+\frac{(1+\Theta)^{2}}{2})$, then it holds for all $c_1\in(0,c_2)$. By Proposition \ref{prop1.5} and a union bound, for any $c,\epsilon>0$,
\begin{equation}\label{main-lem4.9}
    \p\left(\max\limits_{1\leq s\leq \text{e}^{c\sqrt{M}}}g(s)\geq M\right)\leq \text{e}^{c\sqrt{M}-(2-\epsilon)c^{-1}\sqrt{M}}
\end{equation}
which implies 
\begin{align}
   h(c)\leq \frac{c}{2}-\frac{1}{c}. \label{eq:h-roughbound}
\end{align}

We now prove Lemma \ref{lem4.3}. First we state an auxiliary result containing the self-boosting method.
\begin{lemma}\label{lem4.4}
    If \eqref{eq:lem4.3} holds for $c\leq c_1$, then it holds for any 
    $$c=c_2<\min\left(1+\frac{1-\Theta}{4}c_1,-2\Theta+\frac{(1+\Theta)^{2}}{2}\right).$$
\end{lemma}
The proof of Lemma \ref{lem4.4} is postponed. Assuming it for the moment, we proceed to prove Lemma \ref{lem4.3}.
\begin{proof}[Proof of Lemma \ref{lem4.3} assuming Lemma \ref{lem4.4}]
Since $\Theta+\frac{(1+\Theta)^2}{4}\geq -1$, it follows from \eqref{eq:h-roughbound} that \eqref{eq:lem4.3} holds for $0<c\leq 1$.
Let $\theta = \inf\big\{s : h(s) > -\frac{s}{2} - \Theta + \frac{(1+\Theta)^2}{4}\big\}$. By Lemma \ref{lem4.4},
\begin{equation*}
    \theta\geq \min\left(1+\frac{1-\Theta}{4}\theta,-2\Theta+\frac{(1+\Theta)^2}{2}\right),
\end{equation*}
which implies $\theta\geq -2\Theta+\frac{(1+\Theta)^2}{2}$. Then we complete the proof of Lemma \ref{lem4.3}.   
\end{proof}

\begin{proof}[Proof of Lemma \ref{lem4.4}]
    Define event
    \begin{equation*}
         \mathcal{G}_{s_1,s_2}:=\left\{ s_1<s\leq s_2 : \exists \ k,\ 1 \leq k \leq \tau^G_{s_2}, \text{ such that } G_k = s, \text{ and } G_l \neq s \text{ for } 1 \leq l \leq \tau^G_{s_2},\ l \neq k \right\}
    \end{equation*}
Then, $\mathcal{G}_{s_2} \subset \mathcal{G}_{s_1} \cup \mathcal{G}_{s_1,s_2}$ for all $s_1<s_2$.
We assume $c_2 > 1$ and let $N_1 = \lfloor e^{(2c_2 - 2)\sqrt{M}} \rfloor$. With this, define $\tilde{\beta}_{1,k}:=\tau_{kN_1+1}^{G}$. Then, for $i \geq 1$, we define the following stopping times:
    \begin{align*}
        &\tilde\gamma_{i,k}:=\inf\{n>\tilde\beta_{i,k}:\ G_n\in \{G_{\tilde\beta_{i,k}}-1,(k+1)N_1\}\};\\    &\tilde\beta_{i+1,k}:=\inf\{n>\gamma_{i,k}':\ G_n=M^G_{\gamma_{i,k}'}+1\}.
    \end{align*}
    Under the condition that $M^G_{\tilde\gamma_{i,k}}-G_{\tilde\beta_{i,k}}=l-1$, we have 
    $$(G_{n+{\tilde\beta_{i,k}}-1}-G_{{\tilde\beta_{i,k}}})_{1\leq n\leq \tau^G_{{G_{\tilde\beta_{i,k}}+l-1}}} \stackrel{d}{=}(G_n)_{1\leq n\leq \tau^G_{l}}.$$ 
    Take $a_k$ such that $G_{\tilde\gamma_{a_k,k}}=(k+1)N_1$, then for any $M'>0$ and $1\leq i\le a_k$, we have
    \begin{equation}
        \p\left(\left.\max_{G_{\tilde\beta_{i,k}}\leq s<G_{\tilde\beta_{i+1,k}}}|\mathcal{G}_{kN_1,s}|\geq M'\right| M^G_{\tilde\gamma_{i,k}}-G_{\tilde\beta_{i,k}}=l-1,\mathcal{F}^G_{\tilde\beta_{i,k}}\right)=\p\left(\max\limits_{1\le s\le l}g(s)\geq M'\right).\label{eq:4.4}
    \end{equation}
    By Lemma \ref{lem4.2}, for $\p\left(\max\limits_{1\leq s\leq N_1}g(s)\geq M'\right)\leq 0.1$,
    \begin{equation}
        \p(g(l)\geq M')\leq 16(\log N_1+1)lN_1^{-1}\p\left(\max_{1\leq s\leq N_1}g(s)\geq M'\right).\label{eq:4.5}
    \end{equation}
Summing \eqref{eq:4.4} over $1 \leq i \leq a_k$, combining with \eqref{eq:4.5} gives
    \begin{align}\label{mathcal-G-calculate}
        &\p\left(\max_{kN_1+1\leq n\leq (k+1)N_1}\mathcal{G}_{kN_1,n}\geq M'\,\Big|\,\mathcal{F}^G_{\tau^G_{kN_1}}\right)\notag\\
        \leq& \sum_{i=1}^{a_k}\p\left(g(M^G_{\tilde\gamma_{i,k}}-G_{\tilde\beta_{i,k}}+1)\geq M'\Big|\mathcal{F}^G_{\tilde\beta_{i,k}}\right)\notag\\
        \leq& 16(\log N_1+1)N_1^{-1}\sum_{i=1}^{a_k}(M^G_{\tilde\gamma_{i,k}}-G_{\tilde\beta_{i,k}}+1)\cdot\p\left(\max_{1\leq s\leq N_1}g(s)\geq M'\right)\notag\\
        =&16(\log N_1+1)\cdot\p\left(\max_{1\leq s\leq N_1}g(s)\geq M'\right).
    \end{align}
For any $\epsilon > 0$, assume that $\max\limits_{kN_1+1 \leq s \leq (k+1)N_1} g(s) \geq M$. Then there exists an integer $0 \leq i \leq \epsilon^{-1}$ with the property that 
\begin{itemize}
    \item $g(kN_1) \geq \epsilon i M$,
    \item $\max\limits_{kN_1+1 \leq s \leq (k+1)N_1} \mathcal{G}_{kN_1,s} \geq (1 - \epsilon(i+1)) M$.
\end{itemize}
 It follows that
    \begin{eqnarray}\label{liminf-cal01}
        &&\p\left(\max_{kN_1+1\leq s\leq (k+1)N_1}g(s)\geq M\right)\notag\\
        &\leq& \sum_{i=0}^{\lfloor \epsilon^{-1}\rfloor} \p(g(kN_1)\geq \epsilon iM)\cdot\p\left(\max_{kN_1+1\leq s\leq (k+1)N_1}\mathcal{G}_{kN_1,s}\geq (1-\epsilon(i+1))M\,\Big|\,\mathcal{F}^G_{\tau^G_{kN_1}}\right)\notag\\
        &\overset{\eqref{mathcal-G-calculate}}{\le} & M\cdot\sum_{i=0}^{\lfloor \epsilon^{-1}\rfloor} \p(g(kN_1)\geq \epsilon iM)\cdot\p\left(\max_{1\leq s\leq N_1}g(s)\geq (1-\epsilon(i+1))M\right).
    \end{eqnarray}
    By Proposition \ref{prop1.5}, for any $0\leq k\leq N_1^{-1}e^{c_2\sqrt{M}}$ and any $\epsilon'>0$, for sufficient large $M$,
    \begin{equation}\label{liminf-cal02}
        \p(g(kN_1)\geq \epsilon iM)\leq \exp\left({-(2-\epsilon')\epsilon i c_2^{-1}\sqrt{M}}\right).
    \end{equation}
    Putting together \eqref{liminf-cal01} and \eqref{liminf-cal02}, we obtain
    \begin{align}
&\limsup_{M\to\infty}\frac{\log\left(\p\left(\max\limits_{kN_1+1\leq s\leq (k+1)N_1}g(s)\geq M\right)\right)}{2\sqrt{M}}\notag\\
\leq& \max\limits_{c\in [0,1]}\left\{-cc_2^{-1}+\limsup_{M\to\infty}\frac{\log\left(\p\left(\max\limits_{1\leq s\leq N_1}g(s)\geq (1-c)M\right)\right)}{2\sqrt{M}}\right\}.\label{eq:4.7}
    \end{align}
    
    {\bf Case 1. }When $0 \leq c \leq \frac{(1+\Theta)}{2}$, it follows that $\frac{(2c_2 - 2)}{\sqrt{1-c} }\leq c_1$. Then
    \begin{eqnarray*}
        &&\limsup_{M\to\infty}\frac{\log\left(\p\left(\max\limits_{1\leq s\leq N_1}g(s)\geq (1-c)M\right)\right)}{2\sqrt{M}}\\
        &=&\limsup_{M\to\infty}\frac{\sqrt{1-c}\log\left(\p\left(\max\limits_{1\leq s\leq \text{e}^{(2c_2-2)\sqrt{\frac{M}{1-c}}}}g(s)\geq M\right)\right)}{2\sqrt{M}}\\
        &\overset{\eqref{main-lem4.9}}{\leq}& \sqrt{1-c}\left(\frac{c_2-1}{\sqrt{1-c}}+\Theta-\frac{(1+\Theta)^2}{4}\right)=-1+c_2+\sqrt{1-c}\left(\Theta-\frac{(1+\Theta)^2}{4}\right)\\
       &\leq& -1+c_2+\Theta-\frac{(1+\Theta)^2}{4}+cc_2^{-1},
    \end{eqnarray*}
    where the last inequality comes from Proposition \ref{prop3.1}.
    \smallskip
    
    {\bf Case 2. }When $\frac{1+\Theta}{2}\leq c\leq 1-\frac{(c_2-1)^2}{\Theta^2}$, for any $l>0$, since
    \begin{align*}
        \Theta=&\limsup_{M\to\infty}\frac{\log(\p(\max\limits_{1\leq m\leq \sigma-1}\curvevisitonenumber{m}\geq M))}{2\sqrt{M}}\\
        \geq& \limsup_{M\to\infty}\left(\frac{\log(\text{e}^{-(l+2h(l))\sqrt{M}})}{2\sqrt{M}}\right)\geq -\inf\left\{l:h\left(2l\right)=0\right\}.
    \end{align*}
    By Lemma \ref{lem4.2}, for $l\leq -2\Theta$, 
    \begin{equation*}
        h(l)\leq \frac{2\Theta+l}{2}.
    \end{equation*}
    Then we see that
    \begin{align*}
        &\limsup_{M\to\infty}\frac{\log\left(\p\left(\max\limits_{1\leq s\leq N_1}g(s)\geq (1-c)M\right)\right)}{2\sqrt{M}}\\
        =&\limsup_{M\to\infty}\frac{\sqrt{1-c}\log\left(\p\left(\max\limits_{1\leq s\leq \text{e}^{(2c_2-2)\sqrt{\frac{M}{1-c}}}}g(s)\geq M\right)\right)}{2\sqrt{M}}\\
        \leq& \sqrt{1-c}\left(\frac{c_2-1}{\sqrt{1-c}}+\Theta\right)=-1+c_2+\Theta\sqrt{1-c}
        \leq -1+c_2+\Theta-\frac{(1+\Theta)^2}{4}+cc_2^{-1}.
    \end{align*}
    
    {\bf Case 3. }When $1-\frac{(c_2-1)^2}{\Theta^2}\leq c\leq 1$,
    \begin{align*}
        &\limsup_{M\to\infty}\frac{\log\left(\p\left(\max\limits_{1\leq s\leq N_1}g(s)\geq (1-c)M\right)\right)}{2\sqrt{M}}\leq 0
        \leq -1+c_2+\Theta-\frac{(1+\Theta)^2}{4}+cc_2^{-1}.
    \end{align*}
    Putting the three cases together, we conclude that
    \begin{equation}
         \min_{-c\in [0,1]}\left\{-cc_2^{-1}+\limsup_{M\to\infty}\frac{\log\left(\p\left(\max\limits_{1\leq s\leq N_1}g(s)\geq (1-c)M\right)\right)}{2\sqrt{M}}\right\}\leq -1+c_2+\Theta-\frac{(1+\Theta)^2}{4}.\label{eq:4.8}
    \end{equation}
    Combining \eqref{eq:4.7} and \eqref{eq:4.8}, we have
    \begin{align*}
        h(c_2)=&\limsup_{M\to\infty}\frac{\log\left(\p\left(\max\limits_{1\leq s\leq \text{e}^{c_2\sqrt{M}}}g(s)\geq M\right)\right)}{2\sqrt{M}}.\\
        \leq &\limsup_{M\to\infty}\frac{\log\left(\sum\limits_{k=0}^{\lfloor\frac{e^{c_2\sqrt{M}}}{N_1}\rfloor}\p\left(\max\limits_{kN_1+1\leq s\leq (k+1)N_1} g(s)\geq M\right)\right)}{2\sqrt{M}}.\\
        \leq &-\frac{c_2}{2}+1-1+c_2+\Theta-\frac{(1+\Theta)^2}{4}=\frac{c_2}{2}+\Theta-\frac{(1+\Theta)^2}{4}.
    \end{align*}
    which completes the proof of Lemma \ref{lem4.4}.
\end{proof}

\section*{Acknowledgement} 

The authors are grateful to Xinyi Li for his insightful comments, which have greatly improved the manuscript.
All authors are partially supported by the National Natural Science Foundation of China (NSFC) through the NSFC Key Program (Grant No.\ 12231002).

\end{document}